\documentclass[12pt]{amsart}
\usepackage{amssymb}
\usepackage{amsmath, graphicx, rotating}
\usepackage{color}
\usepackage{soul}
\usepackage{todonotes}

\usepackage{dsfont}
\usepackage{mathrsfs}
\usepackage[T1]{fontenc}
\usepackage{lmodern}

\usepackage[english]{babel}

\usepackage{ upgreek }
\usepackage{amsmath,amssymb}
\usepackage{amsthm}
\usepackage{graphicx}
\usepackage{float}

\usepackage{ bbm }
\usepackage{ stmaryrd }
\usepackage{ mathrsfs }
\usepackage{frcursive}
\usepackage{calc}

\usepackage{pgf, tikz}
\usepackage{enumitem}

\definecolor{rouge}{rgb}{0.7,0.00,0.00}
\definecolor{vert}{rgb}{0.00,0.5,0.00}
\definecolor{bleu}{rgb}{0.00,0.00,0.8}

\newtheorem{theorem}{Theorem}[section]
\newtheorem{lemma}[theorem]{Lemma}

\newtheorem{corollary}[theorem]{Corollary}
\newtheorem{remark}[theorem]{Remark}

\numberwithin{equation}{section}

\def\bb#1{\mathbb{#1}}

\def\geq{\geqslant}
\def\leq{\leqslant}

\def\l{\left(}
\def\r{\right)}
\def\p{\mathbb{P}}
\def\e{\mathbb{E}}

\def\({\left(}
\def\){\right)}

\begin{document}
\title[Asymptotic of the distribution for BPREs]{Asymptotic of the distribution and harmonic moments for a supercritical branching process in a random environment}

\author{Ion~Grama}
\curraddr[Grama, I.]{ Universit\'{e} de Bretagne-Sud, LMBA UMR CNRS 6205,
Vannes, France}
\email{ion.grama@univ-ubs.fr}

\author{Quansheng Liu}
\curraddr[Liu, Q.]{ Universit\'{e} de Bretagne-Sud, LMBA UMR CNRS 6205,
Vannes, France}
\email{quansheng.liu@univ-ubs.fr}

\author{Eric Miqueu}
\curraddr[Miqueu, E.]{Universit\'{e} de Bretagne-Sud, LMBA UMR CNRS 6205,
Vannes, France}
\email{eric.miqueu@univ-ubs.fr}
\date{\today }

\begin{abstract}
Let $(Z_n)$ be a supercritical branching process in an independent and identically distributed random environment $\xi$. 
We show the exact decay rate of the probability $\mathbb{P}(Z_n=j | Z_0 = k)$
as $n \to \infty$, for each $j \geq k,$ assuming that $\p (Z_1 = 0) =0$. 
We also determine the critical value for the existence of harmonic moments of the random variable $W=\lim_{n\to\infty}\frac{Z_n}{\mathbb E (Z_n|\xi)}$ under a simple moment condition.
\vskip2mm
\noindent {\sc R\' esum\'e.} 
Soit $(Z_n)$ un processus de branchement surcritique en environnement al\'eatoire  $\xi$ ind\'ependant et identiquement distribu\'e. 
Nous donnons un \'equivalent de la probabilit\'e $\mathbb{P}(Z_n=j | Z_0 = k)$
lorsque $n \to \infty$, pour tout $j \geq k,$ sous la condition $\p (Z_1 = 0) =0$. 
Nous d\'eterminons \'egalement la valeur critique pour l'existence des moments harmoniques de la variable al\'eatoire limite $W=\lim_{n\to\infty}\frac{Z_n}{\mathbb E (Z_n|\xi)}$, sous une hypoth\`ese simple d'existence de moments.
\end{abstract}


\maketitle

\section{Introduction}

A branching process in a random environment (BPRE) is a natural and important generalisation
of the Galton-Watson process, where the reproduction law varies according to a random environment indexed by time.
It was introduced for the first time in Smith and Wilkinson \cite%
{smith} to modelize the growth of a population submitted to an environment.
For background concepts and basic results concerning a BPRE we refer to Athreya and Karlin \cite%
{athreya1971branching, athreya1971branching2}.
In the critical and subcritical regime the branching process goes out and the research interest is mostly concentrated  on the survival probability
and conditional limit theorems,
see e.g. Afanasyev, B\"oinghoff, Kersting, Vatutin \cite{afanasyev2012limit, afanasyev2014conditional}, Vatutin \cite{Va2010},  Vatutin and Zheng \cite{VaZheng2012}, and the references therein.
In the supercritical
case, a great deal of current research has been focused on large deviations, 
see Bansaye and Berestycki \cite{bansaye2009large}, Bansaye and B\"oinghoff \cite{bansaye2011upper, bansaye2013lower, bansaye2014small}, B\"oinghoff and Kersting \cite{boinghoff2010upper}, Huang and Liu \cite{liu}, Nakashima \cite{Nakashima2013lower}. 
In the particular case when the offspring distribution is geometric, precise asymptotics can be found in 
B\"oinghoff \cite{boinghoff2014limit}, Kozlov \cite{kozlov2006large}.

An important closely linked issue is the asymptotic behavior of the distribution of a BPRE $(Z_n),$ 
 i.e. the limit  of $\mathbb{P} ( Z_n = j | Z_0=k )$ as $n \to \infty$, for fixed $j \geq 1$ when the process starts with $k\geq 1$
 initial individuals. 
For the Galton-Watson process, the asymptotic  is well-known and can be found in the book by Athreya \cite{athreya}.
 For the need of the lower large deviation principle of a BPRE, Bansaye and B\"oinghoff have shown in \cite{bansaye2014small} that, 
 for any fixed $j\geq 1$ and $k \geq 1$ it holds $n^{-1} \log \p ( Z_n = j | Z_0 = k) \to -\rho $ as $ n \to \infty$, where $\rho >0$ is a constant. This result characterizes  the exponential decrease of the probability $\p ( Z_n = j | Z_0 = k)$ for the general supercritical case, when extinction can occur. 
 However, it stands only on a logarithmic scale, and the constant $\rho$ is not explicit, 
except when the reproduction law is fractional linear, for which $\rho$ is explicitly computed in \cite{bansaye2014small}. 
Sharper asymptotic results for the fractional linear case can be found in \cite{boinghoff2014limit}.
In the present paper, we improve the results of \cite{bansaye2014small} and extend those of  \cite{boinghoff2014limit} 
 by giving an equivalent of the probability $\p ( Z_n =j | Z_0 =k)$ as $n \to \infty$, 
 provided that each individual gives birth to at least one child. 
 These results are important to understand the asymptotic law of the process, and are useful to obtain sharper asymptotic large deviation results.
 We also improve the result of \cite{liu} about  the critical value for the harmonic 
 moment of the limit variable $W=\lim_{n\to\infty}\frac{Z_n}{\mathbb E (Z_n|\xi)}.$  
 
Let us explain briefly the findings of the paper. 
Assume that $\p (Z_1=0)=0.$
From Theorem \ref{thm small value probability 2} of the paper it follows that when $Z_0=1,$
\begin{equation}
\label{into-eq small val prob}
\mathbb{P} \left( Z_n = j \right) \underset{n \to \infty}{\sim} \gamma ^n q_{j} \quad \text{with } \quad \gamma=\p (Z_1=1)>0,
\end{equation}
where $q_{j} \in (0, + \infty )$ can be computed as the unique solution of some recurrence equations;
moreover, the generating function $Q(t)=\sum_{j=1}^{\infty} q_j t^j$ has the radius of  convergence equal to $1$ 
and is characterized by the functional equation
\begin{equation} \label{intro-eq func Q}
\gamma Q(t) = \mathbb E Q(f_0 (t)),\quad t \in [0,1),
\end{equation}
where $f_0(t)=\sum_{i=1}^{\infty} p_i ( \xi_0 ) t^i$ is the conditional generating function of $Z_1$ given the environment.  
These results extend the corresponding results for the Galton-Watson process (see \cite{athreya}).
They also improve and complete the results in \cite{bansaye2014small} and \cite{boinghoff2014limit}:  it was proved in \cite{bansaye2014small}  
that $\frac{1}{n}\log \mathbb{P} \left( Z_n = j \right) \to \log \gamma,$ and in \cite{boinghoff2014limit} that 
$\mathbb{P} \left( Z_n = 1 \right) \underset{n \to \infty}{\sim} \gamma ^n q_{1}$ in the  fractional linear case.

In the proofs of the above results we make use of Theorem \ref{thm harmonic moments W} 
which shows that,
with $m_0=\e _{\xi} Z_1,$ we have, for any fixed $a>0,$ 
\begin{equation}\label{intro-eq mom harm}
\mathbb{E} W^{-a} < \infty \quad \text{if and only if} \quad \mathbb{E} \left[ p_1 (\xi_0) m_0^a \right] <1,
\end{equation}
under a simple moment condition $\mathbb{E} \left[ m_0^{p} \right] < \infty $ for some $p >a,$ 
which is much weaker than the boundedness condition  
used in \cite[Theorem 1.4]{liu} (see \eqref{condition H} below).



For the proof of  Theorem \ref{thm harmonic moments W} our argument consists of two steps. 
In the first step we prove the existence of the harmonic moment of some order $a>0$
using the functional relation \eqref{eq relation phi xi 1}. 
The key argument to approach  the critical value is in the second step, which is based on the method developed in 
\cite[Lemma 4.1]{liu1999asymptotic}
 obtaining the decay rate of 
the Laplace transform $\phi(t)=\e e^{-tW}$ as $t \to \infty,$ starting from a functional inequation of the form
\begin{equation}
\phi (t) \leq q \mathbb{E} \phi ( Yt) + C t^{-a},
\label{basic001}
\end{equation}
where $Y$ is a positive random variable. 
To prove \eqref{basic001} we use a recursive procedure for branching processes starting with $k$
individuals and choosing $k$ large enough.  
The intuition behind this consideration is that as the number of starting individuals $k$ becomes larger, the 
decay rate of $\phi_k(t)=\e \left[ e^{-tW} | Z_0=k \right]$ as $t \to \infty$ is higher which leads to the desired functional inequation.  

In the proof of Theorem \ref{thm small value probability 2}, the equivalence relation \eqref{into-eq small val prob}
and the recursive equations for the limit values $(q_j)$ come from simple monotonicity arguments. 
The difficulty is to characterize the sequence $(q_j)$ by its generating function $Q$. 
To this end, we first calculate the radius of convergence of $Q$  by
determining the asymptotic behavior of the normalized  harmonic moments $\e Z_n^{-r}/\gamma^{n}$ as $n \to\infty$ for some $r>0$
large enough
and by using the fact that $\sum_{j=1}^{\infty} j^{-r} q_j = \lim_{n\to \infty} \e Z_n^{-r}/\gamma^{n}.$ 
We then show that the functional equation  \eqref{intro-eq func Q} has a unique solution
subject to an initial condition.

 The rest of the paper is organized as follows. 
 The main results,  Theorems \ref{thm harmonic moments W} and \ref{thm small value probability 2},
are presented in Section \ref{secMain}.
Their proofs are given in Sections \ref{sec harmonic moments W} and \ref{sec small value non extinction}.


\section{Main results}\label{secMain}

A BPRE $(Z_n)$ can be described as follows.
The random environment is
represented  by a sequence $\xi = (\xi_0, \xi_1 , ... ) $ of independent and
identically distributed random variables (i.i.d.\ r.v.'s), whose
realizations determine the probability generating functions
\begin{equation}
 f_n (t) = \mathnormal{f} (\xi_n ,t) = \sum_{i=0}^{\infty} p_i ( \xi_n ) t^i,
\quad t \in [0,1], \quad p_i ( \xi_n ) \geq 0, \quad \sum_{i=0}^{ \infty}
p_i (\xi_n) =1.  \label{defin001}
\end{equation}
The branching process $(Z_n)_{n \geq 0}$ is defined by the relations
\begin{equation}  \label{relation recurrence Zn}
Z_0 = 1, \quad Z_{n+1} = \sum_{i=1}^{Z_n} N_{n, i}, \quad \text{for} \quad n
\geq 0,
\end{equation}
where $N_{n,i} $ is the number of children of the $i$-th individual of the
generation $n$. Conditionally on the environment $\xi $, the r.v.'s $N_{n,i} $
(i = 1, 2, ...) are independent of each other with common probability generating function $\mathnormal{f}_n,$
and also independent of $Z_n$. 

In the sequel we denote by $\mathbb{P}_{\xi}$ the
\textit{quenched law}, i.e.\ the conditional probability when the
environment $\xi$ is given, and by $\tau $ the law of the environment $\xi$.
Then
$\mathbb{P}(dx,d\xi) = \mathbb{P}_{\xi}(dx) {\tau}(d\xi)$
is the total law of the process, called
\textit{annealed law}. The corresponding quenched and annealed expectations
are denoted respectively by $\mathbb{E}_{\xi}$ and $\mathbb{E}$.  We also denote by $\mathbb{P}_k$ and $\mathbb{E}_k$ the corresponding probability and expectation starting with $k$ individuals.  
For $n \in \mathbb{N}$, the probability generating function of $Z_n$ is
\begin{equation}
\label{Gn}
G_n (t) = \mathbb{E} t^{Z_n} = \mathbb{E} \left[ f_0 \circ \ldots \circ f_{n-1} (t) \right] = \mathbb{E} \left[ g_n (t) \right],
\end{equation}
where $g_n (t) = f_0 \circ \ldots \circ \ f_{n-1} (t)$ is the conditional probability generating function of $Z_n$ when the environment $\xi$ is given.
It follows from \eqref{relation recurrence Zn} that the probability generating function $G_{k,n}$ of $Z_n$ starting with $k$ individuals is
\begin{equation}
\label{Gkn}
G_{k,n} (t) = \mathbb{E}_k t^{Z_n} = \mathbb{E} \left[ g_n^k (t) \right].
\end{equation}
We also
define, for $n\geq 0$,
\begin{equation*}
m_n = m ( \xi_n )= \sum_{i=0}^\infty i p_i ( \xi_n ) \quad \text{and} \ \ \Pi_n = \mathbb{E}_{\xi} Z_n = m_0 ... m_{n-1},
\end{equation*}
where $m_n $ represents the average number of children of an
individual of generation $n$  when the environment $\xi $ is given. Let
\begin{equation} \label{Wn}
W_n =\frac{Z_n}{\Pi_n} , \quad n\geq 0,
\end{equation}
be the normalized population size.
It is well known that under 
$\mathbb{P}_{\xi},$ as well as under $\mathbb{P},$ the sequence 
$(W_n)_{n \geq 0} $ is a non-negative martingale with respect to the filtration
$$\mathcal{F}_n = \sigma
\left(\xi, N_{j,i} , 0 \leq j \leq n-1, i = 1,2 \ldots \right), $$
where by convention $\mathcal{F}_0 = \sigma(\xi)$.
Then the limit $W = \lim_{n\to \infty} W_n $ exists $\mathbb{P}$ - a.s. and $\mathbb{E} W \leq 1 $.


We shall assume that 
\begin{equation*}
\mu := \mathbb{E}  \log m_0  \in (0, \infty ), 
\end{equation*}
which implies that the BPRE is supercritical
and  that 
\begin{equation}
\gamma := \mathbb{P} ( Z_1=1) \in [0,1).
\end{equation}
With the extra condition $\mathbb E |\log (1-p_0(\xi_0))| <\infty$ (see \cite{smith}), 
the population size tends to infinity with positive probability.
We also assume in the whole paper that each individual gives birth to at least one child, i.e. 
\begin{equation}
\label{condition p0=0}
p_0(\xi_0) = 0 \quad a.s.
\end{equation}
Therefore, under the condition 
\begin{equation}
\label{CN CV L1 W}
\mathbb{E} \frac{Z_1}{m_0} \log^+ Z_1  < \infty ,
\end{equation}
the martingale $(W_n)$  converges to $W$ in $L^1 (\mathbb{P})$ (see e.g. \cite{tanny1988necessary}) and
\[ \mathbb{P} (W>0) =  \mathbb{P} (Z_n \to \infty) 
=1. \]

Our first result concerns the harmonic moments of the r.v.\ $W$. 
\begin{theorem}
\label{thm harmonic moments W}
Assume 
that there exists a constant $p>0$ such that $\mathbb{E} \left[ m_0^{p} \right] < \infty $. Then for any $a\in (0,p)$,
\begin{equation*}
\mathbb{E}_k W^{-a} < \infty \quad \text{if and only if} \quad \mathbb{E} \left[ p_1^k (\xi_0) m_0^a \right] <1.
\end{equation*}
\end{theorem}
From Theorem \ref{thm harmonic moments W} we get the following corollary.
\begin{corollary}
\label{cor harmonic moments W}
Let $a_k>0$ be the solution of the equation 
\begin{equation}
\label{eq moment harmonique critique ak}
\e [ p_1^k m_0^{a_k} ] =1 .
\end{equation}
Assume 
that $\mathbb{E}m_0^{a_k} < \infty$. Then,
\begin{equation*}
\left\{ 
\begin{array}{l}
 \mathbb{E}_k W^{-a} < \infty \quad \text{for}\quad  a \in [0, a_k), \\
 \mathbb{E}_k W^{-a} =\infty \quad \text{for}\quad  a \in [a_k, \infty). 
\end{array}
\right.
\end{equation*}
\end{corollary}
The solution $a_k$ of the equation  \eqref{eq moment harmonique critique ak}
is the critical value for the existence of harmonic moments of the r.v.\ $W$. Note that, when the process starts with one individual, 
the critical value $a_1$ for the harmonic moments of $W$
has been found in Theorem 1.4 of \cite{liu} under the more restrictive condition
\begin{equation}
\label{condition H}
 A_1 \leq m_0 \quad \text{and} \quad
\sum_{i=1}^\infty i^{1+ \delta} p_i (\xi_0)  \leq A^{1+ \delta} \quad a.s.,
\end{equation}
where $\delta>0$ and $1<A_1 <A$ are some constants. 
Theorem \ref{thm harmonic moments W} and Corollary \ref{cor harmonic moments W} 
generalize the result of \cite{liu}, 
in the sense that
we consider $k$ initial individuals rather than just one and that 
the boundedness condition \eqref{condition H}
is relaxed to the simple moment condition $\mathbb{E} \left[ m_0^{p} \right] < \infty $.

The next result gives an equivalent as $n \to \infty$ of the probability $\mathbb{P}_k \l Z_n = j \r = \mathbb{P} \l Z_n = j | Z_0=k \r$, with $k \in \mathbb{N}^*$ and $j \geq k,$ in the case when  $\p(Z_1=1)>0$. 
The last condition implies that, 
for $k \geq 1$, 
\begin{equation}
\label{eq gamma_k}
\gamma_k = \mathbb{P}_k (Z_1=k) = \mathbb{E} [ p_1^k (\xi_0) ] >0.
\end{equation}
Define $r_k$ as the solution of the equation
\begin{equation}
\label{eq rk}
\gamma_k = \e m_0^{-r_k}.
\end{equation}
\begin{theorem} Assume that $\p(Z_1=1)>0$. For any $k\geq 1$  the following assertions holds.
\begin{enumerate}[ref=\arabic*, leftmargin=*, label=\arabic*.]
\item[a)] \label{thm small value probability 2}
 For any accessible state $ j \geq k$ in the sense that $\mathbb{P}_k (Z_l=j)>0$ for some $l \geq 0$, we have  
\begin{equation}
\label{small value asymptotic 2}
\mathbb{P}_k \left( Z_n = j \right) \underset{n \to \infty}{\sim} \gamma_k^n q_{k,j},
\end{equation}
where $q_{k,k}= 1$ and, for  $j>k$,  $q_{k,j} \in (0, + \infty )$ is the solution of the recurrence relation
\begin{equation}
\label{relation rec qkj}
\gamma_k q_{k,j} = \sum_{i=k}^j p(i, j) q_{k, i},
\end{equation}
with $q_{k,i}=0$ for any non-accessible state $i$, i.e.\  $\mathbb{P}_k (Z_l=i)=0$ for all $l \geq 0$.

\item[b)] Assume 
that  there exists $\varepsilon>0$ such that $\mathbb{E} [ m_0^{r_k+ \varepsilon} ] < \infty $. Then, for any $r>r_k$, we have $ \sum_{j=k}^{\infty} j^{-r} q_{k,j} < \infty$. In particular the radius of convergence of the power series
\begin{equation}
Q_k (t) = \sum_{j=k}^{+ \infty} q_{k,j} t^j
\end{equation}
is equal to 1.  

\item[c)] For all $t \in [0, 1)$ and $k \geq 1$, we have,
\begin{equation}
\label{cv Qnk ->Qk}
 \frac{G_{k, n} (t)}{\gamma^n_k} \uparrow Q_k(t) \ \ \text{as} \ \ n \to \infty,
\end{equation}
where $G_{k,n}$ is the probability generating function of $Z_n$ when $Z_0=k$, defined in \eqref{Gkn}.

\item[d)] $Q_k (t)$ is the unique power series which verifies the functional equation
\begin{equation}
\label{relation Q_k}
\gamma_k Q_k (t) = \mathbb{E} \left[ Q_k ( f_0 (t) ) \right], \ \ t\in [0,1),
\end{equation}
with the condition $Q_k^{(k)} (0) = 1.$
\end{enumerate}
\end{theorem}
Part a)  improves the bound 
$\mathbb{P} \left( Z_n \leq j \right) \leq n^j \gamma^n $ obtained in \cite{bansaye2009large} (Lemma 7) for a BPRE with $\mathbb{P}(Z_1=0)=0$. 
Furthermore, Theorem \ref{thm small value probability 2} extends the results of \cite{athreya} for the Galton-Watson process, with some significant differences. Indeed, 
when the environment is random and non-degenerate, we have, for $k\geq 2,$ 
$ G_{k,1} (t) = \mathbb{E} f_0^k (t) \neq G_1^k (t)$  in general, which implies that $Q_k (t) \neq Q^k (t)$, whereas we have the relation $Q_k (t) = Q^k (t)$ for the Galton-Watson process.

Theorem \ref{thm small value probability 2} also improves the results of \cite{bansaye2014small} (Theorem 2.1), where it has been proved that for a general supercritical BPRE
\begin{equation}
\label{const-rho-001}
\lim_{n\to\infty}\frac{1}{n} \log \mathbb{P}_k \left( Z_n = j \right) = -\rho <0.
\end{equation}
Our result is sharper in the  case where $\mathbb{P} \left( Z_1 = 0 \right) =0$. Moreover, in the  case where $\mathbb{P} \left( Z_1 = 0 \right) =0$, it has been stated mistakenly in \cite{bansaye2014small} that 
$\lim_{n\to\infty} \frac{1}{n} \log \mathbb{P}_k \left( Z_n = j \right) =  k \log \gamma$, whereas the correct asymptotic is 
$$\lim_{n\to\infty}\frac{1}{n} \log \mathbb{P}_k \left( Z_n = j \right) = \log \gamma_k.$$

Now we discuss the particular fractional linear case. The reproduction law of a BPRE is  said to be fractional linear if 
\begin{equation}
\label{loi linéaire fractionnaire}
p_0 (\xi_0) = a_0, \quad p_k (\xi_0) = \frac{(1-a_0)(1-b_0)}{b_0} b_0^k,
\end{equation}
with generating function $f_0$ given by
\begin{equation*}
\label{eq f cas lineraire fractionnaire}
f_0(t) = a_0 + \frac{(1-a_0)(1-b_0)t}{1-b_0t},
\end{equation*}
where $a_0\in [0,1),$ $b_0 \in  (0,1)$, with $a_0+b_0 \leq 1 $, are random variables depending on the environment $\xi_0$. 
In this case, the mean of the offspring distribution is given by
\begin{equation*}
\label{m_0 lineraire fractionnaire}
m_0 = \frac{1-a_0}{1-b_0}.
\end{equation*}
The constant $\rho$ in \eqref{const-rho-001} was computed in \cite{bansaye2014small}: with $X= \log m_0,$
\begin{equation*}
\label{rho cas LF}
\rho = \left\{ 
\begin{array}{l l l l}
  - \log \e [ e^{-X} ]  & \text{if}& \ \e [ X e^{-X} ] \geq 0 & (\text{intermediately and}\\
  &&& \ \ \text{strongly supercritical case}), \\
  - \log \inf_{\lambda \geq 0} \e [ e^{-X} ]  & \text{if}& \ \e [ X e^{-X} ] < 0 & (\text{weakly supercritical case)}.
\end{array}
\right.
\end{equation*}
Moreover, precise asymptotic results for the strongly and intermediately supercritical case can be found in \cite{boinghoff2014limit}, 
where the following assertions are proved:
\begin{enumerate}
\item if $\e [X e^{-X} ] > 0$ (strongly supercritical case),
\[
\p (Z_n=1) \sim  \nu  \left( \e [e^{-X} ] \right)^n;
\]
\item if $\e [X e^{-X} ] = 0$ (intermediately supercritical case), 
\[
\p (Z_n=1) \sim  \theta  \left( \e [e^{-X} ] \right)^n l(n) n^{-(1-s)},
\]
\end{enumerate}
with $\theta$, $\nu$, $s$ positive constants and $l(\cdot)$ a slowly varying function. 
In the particular case where $a_0=0$, 
Theorem \ref{thm small value probability 2} recovers Theorem 2.1.1 of \cite{boinghoff2014limit} 
with 
$p_1 (\xi_0) = 1/m_0$, $X= \log m_0 >0$ and $\e \left[X e^{-X} \right] >0$. Therefore the process is strongly supercritical and
$ 
\p (Z_n=1) \sim  \nu  \left( \e [e^{-X} ] \right)^n =  \gamma^n.
$
However,
since we assume $\p(Z_1=0)=0$, our result does not highlight the previous two asymptotic regimes stated in the particular case when the distribution is fractional linear.  The study of the  general case is a challenging problem which still remains open.

\section{Harmonic moments of $W$}
\label{sec harmonic moments W}

In this section we prove Theorem \ref{thm harmonic moments W}. 
Denote the quenched Laplace transform of $W$ under the environment $\xi$ by
\begin{equation}
\label{quenched laplace Wn W}
\phi_{\xi} (t) = \mathbb{E}_{\xi} \left[ e^{-t W} \right],
\end{equation}
and the annealed Laplace transform of $W$ starting with $k$ individuals by
\begin{equation}
\label{annealed laplace Wn W}
\phi_{k} (t) = \mathbb{E}_{k} \left[\phi_{\xi} (t) \right] = \mathbb{E} \left[\phi_{\xi}^k (t) \right] = \mathbb{E}_{k} \left[ e^{-t W} \right].
\end{equation}
We start with a lemma which gives a lower bound for the harmonic moment of $W$.
\begin{lemma}
\label{lem ak min}
Assume that $\e \left[ m_0^{p} \right] < \infty$ for some constant $p>0$. For any $k \geq 1$, let
\begin{equation}
\label{eq ak min}
\alpha_k =  \frac{ p  }{ 1 - \log \mathbb{E} m_0^{p} / \log \gamma_k },
\end{equation}
with the convention that $\alpha_k=p$ if  $p_1 (\xi_0)=0$ a.s. (so that $\gamma_k=0$).
Then, for all $a \in (0, \alpha_k)$, 
\[ \mathbb{E}_k W^{-a} < \infty . \]
Furthermore,  if $\p(p_1 (\xi_0)=0)<1$, we have $\alpha_k < \alpha_{k+1}$ ; if additionally 
$\p \left( p_1(\xi_0)<1 \right)=1$, then $\lim_{k \to \infty} \alpha_k = p.$
\end{lemma}
\begin{proof} 
We use the same approach as in \cite{glm2016berry} where the case $k=1$ was treated.
Since $W$ is a positive random variable, it can be easily seen that, for $\alpha >0$,
\begin{equation}
\label{moment harmonique et Laplace}
\e_k W^{- \alpha} = \frac{1}{\Gamma (\alpha)} \int_0^{+ \infty} \phi_k (t)  t^{\alpha-1} dt,
\end{equation}
where $\Gamma(\alpha )=\int_0^\infty t^{\alpha-1}e^{-t}dt$ is the Gamma function.
Moreover, it is well-known that $\phi_{\xi} (t) $ satisfies the functional relation
\begin{equation} \label{eq relation phi xi 1}
\phi _{\xi} (t) = \mathnormal{f}_0 \left( \phi_{T \xi} \left( \frac{t}{m_0} \right) \right),
\end{equation}
where $f_0 (t) = \sum_{k=1}^{\infty} p_k (\xi_0) t^k$ is the generating function of $Z_1$ under $\xi_0$, defined in \eqref{defin001}.
Using \eqref{eq relation phi xi 1} and the fact that $ \phi^k_{T \xi} \left( \frac{t}{m_0} \right) \leq \phi^2_{T \xi} \left( \frac{t}{m_0} \right)$ for all $k \geq 2$, we obtain
\begin{equation}
\label{eq relation phi xi 2}
\phi_{\xi} (t) \leq p_1 (\xi_0) \phi_{T \xi } \left( \frac{t}{m_0} \right) + ( 1 - p_1 ( \xi_0 ) ) \phi_{T \xi }^2 \left( \frac{t}{m_0} \right).
\end{equation}
Taking the $k$-th power in \eqref{eq relation phi xi 2}, using the binomial expansion and the fact that $\phi_{T \xi}^{2k-i} \left( \frac{t}{m_0} \right) \leq \phi_{T \xi}^{k+1} \left( \frac{t}{m_0} \right)$ for all $i \in \{ 0, \ldots, k-1 \}$, we get
\begin{eqnarray}
\label{eq relation phi xi 3}
\phi_{\xi}^{k} (t) &=&  p_1^{k} (\xi_0) \phi_{T \xi}^{k} \left( \frac{t}{m_0} \right) + \sum_{i=0}^{k-1} C^i_{k}\ p_ 1 (\xi_0)^i (1-p_1 (\xi_0))^{k-i} \phi_{T \xi}^{2(k-i)+i}\left( \frac{t}{m_0} \right) \notag \\
& \leq &  p_1^{k} (\xi_0) \phi_{T \xi}^{k} \left( \frac{t}{m_0} \right) + (1-p_1^k(\xi_0))\phi_{T \xi}^{k+1} \left( \frac{t}{m_0} \right)  \notag \\
&=&  \phi_{T \xi}^{k} \left( \frac{t}{m_0} \right) \left[ p_1^{k} (\xi_0)  + (1-p_1^k(\xi_0)) \phi_{T \xi} \left( \frac{t}{m_0} \right) \right].
\end{eqnarray}
By iteration, this leads to
\begin{equation} \label{2.6}
\phi_{\xi}^k (t) \leq \phi_{T^n \xi}^k \left(\frac{t}{\Pi_n}\right) \ \prod_{j=0}^{n-1} \left( p_1^k ( \xi_j) + (1- p_1^k ( \xi_j) ) \phi_{T^n \xi} \left( \frac{t}{\Pi_n} \right) \right).
\end{equation}
Taking expectation and using the fact that  $ \phi_{T^n \xi} (\cdot)  \leq 1$, we have
\[ \phi_k (t) \leq \mathbb{E} \left[ \prod_{j=0}^{n-1} \left( p_1^k ( \xi_j) + (1- p_1^k ( \xi_j) ) \phi_{T^n \xi} \left( \frac{t}{\Pi_n} \right) \right) \right]. \]
Since $ \phi_{\xi} ( \cdot ) $ is non-increasing, using a truncation, we get for all $A>1$,
\begin{eqnarray*}
\phi_k (t) 
&\leq& \mathbb{E} \left[  \prod_{j=0}^{n-1} \left( p_1^k ( \xi_j) + (1- p_1^k ( \xi_j) ) \phi \left( \frac{t}{A^n} \right) \right)  \right] + \mathbb{P}( \Pi_n \geq A^n ).
\end{eqnarray*}
As  $T^n \xi$ is independent of $\sigma ( \xi_0, ... ,\xi_{n-1} )$, and the r.v.'s $ p_1 ( \xi_i)$  ($i\geq 0$) are i.i.d., we obtain
\begin{equation*}
\phi_k (t)  \leq \left[ \gamma_k + (1-  \gamma_k )  \phi \left(\frac{t}{A^n}\right) \right]^n + \mathbb{P}( \Pi_n \geq A^n ),
\end{equation*}
where $\gamma_k = \e p_1^k (\xi_0)$ is defined in \eqref{eq gamma_k}.
By the dominated convergence theorem, we have $\lim_{t \to \infty} \phi (t) = 0$. Thus, for any $\delta \in (0,1)$, there exists a constant $K>0$ such that, for all $ t \geq K$, we have $\phi (t) \leq \delta$.
Consequently, for all
$ t \geq K A^n,$ we have $ \phi  \left(\frac{t}{A^n}\right) \leq \delta $ and 
\begin{equation}
\phi_k (t) \leq \beta^n + \mathbb{P}( \Pi_n \geq A^n ) ,
\label{alpha0}
\end{equation}
where
\begin{equation}
\beta =  \gamma_k + (1-  \gamma_k ) \delta  \in (0,1).
\label{alpha}
\end{equation}
Using Markov's inequality, we have $\mathbb{P} ( \Pi_n \geq A^n) \leq  \left( \mathbb{E} m_0^{p}/ A^{p} \right)^n $.
Setting $ A = \left(\frac{\mathbb{E} m_0^{p}}{\beta} \right)^{1/ p} >1, $ we get for any $ n \in \mathbb{N} $ and $ t \geq K A^n$,
\begin{equation}
 \phi_k (t) \leq 2 \beta^n.
\label{bbb001}
\end{equation}
Now, for any  $t \geq K$, define $n_0 = n_0 (t) = \left[ \frac{\log (t/K)}{\log A } \right] \geq 0$, where $ [x]$ stands for the integer part of $x$, so that
\[  \frac{\log (t/K)}{\log A} - 1 \leq n_0 \leq \frac{\log (t/K)}{\log A} \ \  \text{and} \ \  t \geq K A^{n_0}  .\]
Then, for $t \geq K$,
\[\phi_k (t) \leq 2 \beta^{n_0} \leq 2 \beta^{-1} (t/K)^{\frac{\log \beta }{\log A}} = C_0 t^{-\alpha}, \]
with $C_0 =  2 \beta^{-1} K^{\alpha}$ and $ \alpha = - \frac{\log \beta}{\log A } > 0 $.
Thus, we can choose a constant $C >0$ large enough, such that, for all $t > 0$,
\begin{equation}
\label{majoration phi}
\phi_k (t) \leq C t^{- \alpha}.
\end{equation}
Furthermore, by the definition of $\beta$, $A$ and $\alpha$, we have
\begin{eqnarray*}
\alpha
&=&   \frac{ p }{1- \log \mathbb{E} m_0^{p} / \log \left(\gamma_k + (1- \gamma_k ) \delta \right) },
\end{eqnarray*}
where $\delta \in (0,1)$ is an arbitrary constant and $\gamma_k=\mathbb{E} p_1^k(\xi_0)$.
When $\delta \rightarrow 0$, we have $\alpha \rightarrow \alpha_k$, so that \eqref{majoration phi} holds for all $\alpha < \alpha_k$, where $\alpha_k$ is defined in \eqref{eq ak min}.
By \eqref{moment harmonique et Laplace} and \eqref{majoration phi}, we conclude that $\e W^{- \alpha} < \infty$ for any $\alpha < \alpha_k$.
Moreover, it is easily seen that if $\p(p_1 (\xi_0)=0)<1$, then $\alpha_k < \alpha_{k+1}$ since $\gamma_{k+1} < \gamma_k$;  if additionally $\p (p_1(\xi_0)<1)=1$, then $\lim_{k \to \infty} \gamma_k =0$ so that $\lim_{k \to \infty}  \alpha_k = p$. 
\end{proof}
The following lemma is the key technical tool to study the exact decay rate of the Laplace transform of the limit variable $W$.  
\begin{lemma}[\cite{liu1999asymptotic}, Lemma 4.1]
\label{lem liu}
Let $\phi : \mathbb{R}_+ \to \mathbb{R}_+$ be a bounded function and let $Y$ be a positive random variable such that for some constants $q \in (0,1)$, $a \in (0, \infty)$, $C>0$ and $t_0 \geq 0$ and all $t>t_0$, 
\[
\phi (t) \leq q \mathbb{E} \phi ( Yt) + C t^{-a}.
\]
If $ q \e \left(Y^{-a} \right) < 1$, then $ \phi (t) = O ( t^{-a} )$ as $t \to \infty$.
\end{lemma}
Now we proceed to prove Theorem \ref{thm harmonic moments W}.
We first prove the necessity. 
Assume that  $\mathbb{E}_k W^{-a} < \infty$ for some $a>0$. We shall show that $\mathbb{E} p_1^k (\xi_0) m_0^a < 1$.
Note that the r.v.\ $W$ admits the well-known decomposition 
\[
W = \frac{1}{m_0} \sum_{i=1}^{Z_1} W{(i)},
\]
where the r.v.'s $W{(i)}$ $(i \geq 1)$ are i.i.d.\ and independant of $Z_1$ under $\p_\xi$, and are also independent of $Z_1$ and $\xi_0$ under $\p$. The conditional probability law of $W(i)$  satisfies $\p_\xi (  W(i) \in \cdot ) = \p_{T \xi} (  W \in \cdot )$. Since $\p_k ( Z_1 \geq k+1)>0$, we have
\begin{equation}
\mathbb{E}_k W^{-a} > \e_k m_0^a  \left(\sum_{i=1}^{Z_1} W{(i)} \right)^{-a} \mathds{1} \{ Z_1 = k \} = \e p_1^k (\xi_0) m_0^a  \ \e_k W^{-a},
\end{equation}
which implies that $\e p_1^k (\xi_0) m_0^a < 1$.

We now prove the sufficiency. Assume that $\mathbb{E} m_0^{p}< \infty$ and $\e p_1^k(\xi_0) m_0^a < 1$ for some $a \in (0,p)$.

We first consider the case where $\mathbb{P}(p_1 (\xi_0)<1)=1$. We prove that $\e_k W^{-a}< \infty$ by showing that $ \phi_k (t) = O \left( t^{-(a+ \varepsilon)} \right)$ as $t \to \infty$, for some $\varepsilon>0$. By Lemma \ref{lem ak min}, there exists an integer $j \geq k$ large enough and a constant $C>0$ such that
\begin{equation}
\label{majoration phi_j}
\phi_j (t) \leq C t^{-(a+ \varepsilon)},  
\end{equation}
with $\varepsilon>0$ and $a+ \varepsilon<p$.
By \eqref{eq relation phi xi 3}, we have
\begin{equation}
\label{eq relation phi xi 4}
\phi_{\xi}^{j-1} (t) \leq  p_1^{j-1} (\xi_0) \phi_{T \xi}^{j-1} \left( \frac{t}{m_0} \right) + \phi_{T \xi}^j \left( \frac{t}{m_0} \right).
\end{equation}
Taking the expectation in \eqref{eq relation phi xi 4}, using \eqref{annealed laplace Wn W}, \eqref{majoration phi_j} and the independence between $\xi_0$ and $T \xi$, we obtain
\begin{eqnarray}
\label{eq relation phi xi 5}
\phi_{j-1} (t) 
&\leq& \mathbb{E} \left[ p_1^{j-1}(\xi_0) \phi_{j-1} \left( \frac{t}{m_0} \right) \right] + C t^{-(a+ \varepsilon)} \notag \\
&=& \gamma_{j-1} \mathbb{E} \left[ \phi_{j-1} (Yt) \right] + C t^{-(a+ \varepsilon)},
\end{eqnarray}
where $\gamma_{j-1}= \mathbb{E} \left[ p_1^{j-1} (\xi_0)\right]<1$ and $Y$ is a positive random variable whose distribution is 
determined by
\[
\mathbb{E} \left[ g(Y) \right] = \frac{1}{\gamma_{j-1}} \mathbb{E} \left[ p_1^{j-1}(\xi_0) g \left( \frac{1}{m_0} \right) \right],
\]
for all bounded and measurable function $g$.
By hypothesis, $\e p_1^k (\xi_0) m_0^a < 1$. 
Then, by the dominated convergence theorem, there exists $\varepsilon >0$ small enough such that $\e p_1^k (\xi_0) m_0^{a+\varepsilon} < 1$, and since $  j-1 \geq k$, we have  $\e p_1^{j-1} (\xi_0) m_0^{a+\varepsilon} \leq \e p_1^k(\xi_0)  m_0^{a+\varepsilon}<1$. Therefore, 
$ \gamma_{j-1} \e [ Y^{-(a+\varepsilon)} ]<1$ and using  \eqref{eq relation phi xi 5} and Lemma \ref{lem liu}, we get $ \phi_{j-1} (t) = O ( t^{-(a+ \varepsilon)} )$ as  $t \to \infty $.
By induction, applying \eqref{eq relation phi xi 4} and \eqref{eq relation phi xi 5} to the functions $\phi_{j-2}, \phi_{j-3}, \ldots , \phi_{k}$ 
and using the same argument as in the proof for $\phi_{j-1}$, we obtain
\begin{equation}
\phi_{k} (t) = O ( t^{-(a+ \varepsilon)} ) \quad \text{as} \quad t \to \infty .
\end{equation}
Therefore, in the case where $\mathbb{P}(p_1 (\xi_0)<1)=1$, we have proved that 
\begin{equation}
\label{eq implication moment harm W cas p1<1}
\e p_1^k (\xi_0) m_0^a < 1 \quad \text{implies} \quad 
\e_k W^{-a} < \infty.
\end{equation}

Now consider the general case where $\mathbb{P}(p_1 (\xi_0)<1)<1.$ 
Denote the distribution of $\xi_0$ by $\tau_0$ and define a new distribution  $\tilde{\tau}_0$ as
\begin{equation}
\tilde{\tau}_0 (\cdot) =  \tau_0 (\cdot | p_1 (\xi_0) <1). 
\end{equation}
Consider the new branching process whose environment distribution is  $\tilde{\tau}= \tilde{\tau_0}^{\otimes \mathbb{N}}$ instead of $\tau = \tau_0^{\otimes \mathbb{N}}$. The corresponding probability and expectation are denoted by $\tilde{\p}(dx,d\xi) = \p_\xi(dx) \tilde{\tau}(d\xi)$ and $\tilde{\e}$, respectively.
Of course $(W_n)$ is still a martingale under $\tilde{\p}$. Moreover, the condition $\tilde{\e} \left[ \frac{Z_1}{m_0} \log^+ Z_1 \right]  = \e \left[ \frac{Z_1}{m_0} \log^+ Z_1 \right] /\p ( p_1 (\xi_0) < 1) < \infty$ implies that $W_n \to W$ in $L^1 (\tilde{\p})$.
Now we show that $\mathbb{E}_k \left[ W^{-a} \right]  \leq \tilde{\mathbb{E}}_k \left[ W^{- a} \right] $.
For $ 0 \leq i \leq n$, denote
\begin{eqnarray*}
A_{i,n} &=& \big\{ ( \xi_0, \ldots, \xi_{ n-1}) \ | \ p_1 ( \xi_{j_1}) = \ldots =  p_1 ( \xi_{j_i}) =1 \ \text{for some}\ 0 \leq j_1 < \ldots < j_i \leq n-1, \\
&&\ \text{and}\ p_1 ( \xi_h) <1 \  \text{for all }\ h \in \{ 0, \ldots n -1\} \backslash \{ j_1, \ldots , j_{i} \} \big\}. 
\end{eqnarray*}
Conditioning by the events $A_{i,n}$ ($i \in \{0, \ldots, n \}$) and using the fact that the r.v.'s $\xi_0, \ldots ,\xi_{ n-1}$ are i.i.d., we obtain, for all $n \in \mathbb{N}$,
\begin{eqnarray}
\label{eq moment harmonique Wn p1<1 et p1=1}
\mathbb{E}_k\left[ W_n^{- a} \right] 
&=& \sum_{i=0}^n \mathbb{E}_k \left[ W_n^{- a} \big| A_{i,n} \right] \p \left( A_{i,n} \right) \notag \\
&=& \sum_{i=0}^{n} \mathbb{E}_k \left[ W_n^{- a} \big| A_{i,n} \right] C^i_n \eta^i (1- \eta)^{n-i},
\end{eqnarray}
with $\eta = \p ( p_1 (\xi_0) =1)$. 
Moreover, using \eqref{relation recurrence Zn}, a straightforward computation leads to the decomposition 
\begin{equation}
\label{decomposition produit Wn}
W_{n} = \prod_{i=0}^{n-1} \eta_i, \quad \text{with} \quad  n \geq 1 \quad \text{and} \quad \eta_i =\frac{1}{Z_i} \sum_{j=1}^{Z_{i}} \frac{N_{i,j}}{m_i}.
\end{equation}
Note that, on the event $\left\{p_1 (\xi_i) =1\right\}$ we have $\eta_i=1$. Therefore, using \eqref{decomposition produit Wn} and the fact that the r.v.'s $\xi_0, \ldots, \xi_{ n-1}$ are i.i.d., we get
\begin{equation}
\label{eq E Wn | Akn = tilde E Wn-k}
\mathbb{E}_k \left[ W_n^{- a} \big| A_{i,n} \right] = \tilde{\e}_k \left[ W_{n-i}^{- a} \right].
\end{equation}
By the convexity of the function $x \mapsto x^{- a}$, we have $ \sup_{n  \geq i} \tilde{\e}_k [W_{n-i}^{-a}] \leq  \tilde{\e}_k W^{-a}$ (see \cite{liu} Lemma 2.1). Thus, by \eqref{eq moment harmonique Wn p1<1 et p1=1} and \eqref{eq E Wn | Akn = tilde E Wn-k}, we obtain
\begin{equation}
\label{eq majoration E W-a leq tilde E W-a}
\mathbb{E}_k \left[ W^{- a} \right]  \leq \tilde{\e}_k \left[ W^{-a} \right].
\end{equation}
Note that, conditioning by the events $\{p_1 (\xi_0)=1 \}$ and $ \{ p_1 (\xi_0)<1 \}$, we have
\[
\e p_1^k (\xi_0) m_0^a = (1- \eta) \tilde{\e} p_1^k (\xi_0) m_0^a  + \eta,
\]
with $\eta = \p ( p_1 (\xi_0) =1)$. So the condition $\e p_1^k(\xi_0) m_0^a <1$ implies that $ \tilde{\e} p_1^k(\xi_0) m_0^a <1$.
Then, by \eqref{eq implication moment harm W cas p1<1} applied under the probability $ \tilde{\p} $, and the fact that $\tilde{\p} ( p_1 (\xi_0) < 1) =1$,  we get $\tilde{\e}_k \left[ W^{-a} \right] < \infty$. Therefore, by \eqref{eq majoration E W-a leq tilde E W-a}, it follows that
\begin{equation}
\e p_1^k (\xi_0) m_0^a < 1 \quad \text{implies} \quad \e_k W^{-a} < \infty,
\end{equation}
which ends the proof of Theorem \ref{thm harmonic moments W}.

\section{Small value probability in the non-extinction case}
\label{sec small value non extinction}
In this section we prove Theorem  \ref{thm small value probability 2}.
We start with the proof of part a).
For $k \geq 1$ and $j \geq k$, define
\begin{equation}
a_{k,n} (j) = \frac{ \p \l   Z_n =j  | Z_0=k \r}{\gamma_k^n},
\end{equation}
with $\gamma_k = \p_k (Z_1=k)$. By the Markov property, we have
\[
 \mathbb{P}_k \left( Z_{n+1} = j \right) \geq \mathbb{P}_k \l Z_1 = k \r \mathbb{P}_k \l Z_n = j \r.
\]
Dividing by $\gamma_k^{n+1}$ leads to 
\begin{equation}
\label{majoration a_k,n (j)}
 a_{k,n+1} (j) \geq a_{k,n} (j) . 
\end{equation}
Therefore, by the monotone ratio theorem, we obtain
\[  \lim_{n \to \infty}  \uparrow a_{k,n} (j) = q_{k,j} \in \bar { \bb R} . \]
We shall prove that $q_{k,j}$ satisfies the properties claimed in the theorem. 
If $j$ is such that $\p_k (Z_n=j) =0$ for any $n \geq 0$, then $a_{k,n}(j)=0$ for any $n  \geq 0$, so that $ \lim_{n \to \infty} a_{k,n}(j)=0=q_{k,j}$.
If there exists $l \geq 0$ such that $\p_k (Z_l=j) >0$, then $q_{k,j} \geq a_{k,l} (j) = \p_k (Z_l=j)/ \gamma_k^l >0$. 

Now we show by induction that for all $j \geq k$, we have 
\[ 
H (j): \quad 
 \underset{n \in \bb N}{\sup} \ a_{k,n} (j) = a_k (j) < \infty. 
 \]
For $j=k$, we have $a_k(k)=1$.
Assume that $j \geq k+1$ and that $H(i)$ is true for all $ k \leq i \leq j-1$. By the total probability formula, we obtain
\[ \frac{\p_k \l Z_{n+1}=j \r}{\gamma_k^{n+1}} = \frac{1}{\gamma_k} \sum_{i=k}^{j} \p_k \l Z_{n+1} = j | Z_n = i \r  \frac{\p_k \l Z_n = i \r}{\gamma_k^n} , \]
which is equivalent to
\begin{equation}
\label{recurrence a_n (k)}
a_{k,n+1} (j) = \frac{1}{\gamma_k} \left[ \sum_{i=k}^{j-1} p(i,j) a_{k,n} (i) + \gamma_{j} a_{k,n} (j) \right]
\end{equation}
with $p(i,j) = \p \l Z_{1} = j | Z_0 = i \r $.
Using the fact that $a_{k,n} (j) \leq a_{k,n+1} (j)$, we get by induction that
\begin{eqnarray}
\label{RDR pi_k}
\sup_{n \in \mathbb{N}} \ a_{k,n+1} (j) (\gamma_k - \gamma_{j})  &\leq& \sum_{i=k}^{j-1} p(i,j) a_k(i) < \infty. \nonumber
\end{eqnarray}
Thus $q_{k,j} < \infty$ for all $j \geq k+1$ and $k\geq 1$. 
Furthermore, taking the limit as $n \to \infty$ in (\ref{recurrence a_n (k)}), leads to the following recurrent relation for $q_{k,j}$:
\begin{equation*}
q_{k,k} = 1, \quad \gamma_k q_{k,j} = \sum_{i=k}^j p(i, j) q_{k,i} \quad (j \geq k+1).
\end{equation*}
This end the proof of part a) of Theorem \ref{thm small value probability 2}. 

Now we prove part b) of Theorem \ref{thm small value probability 2}.
We give a proof that the radius of convergence of the power series $Q_k$ is equal to 1. 
The method  is new even in the case of the Galton-Watson process. 
We start with a lemma.
\begin{lemma}
\label{lem R Q}
Let $k \geq 1$. Assume that $\p (p_1(\xi_0) <1 ) =1$ and that there exists $\varepsilon>0$ such that $\mathbb{E}[ m_0^{r_k+ \varepsilon} ] < \infty$, where $r_k$  is the solution of the equation $\gamma_k = \e m_0^{-r_k}$. Then, for any $r>r_k$, we have
\begin{equation}
\label{eq lem R Q}
\lim_{n \to \infty} \uparrow \frac{\mathbb{E}_k Z_n^{-r}}{\gamma_k^n}  < \infty.
\end{equation}
\end{lemma}
\begin{proof}
By the Markov property, 
$$
\mathbb{E}_k \left[ Z_{n+1}^{-r} \right] \geq \mathbb{E}_k \left[ Z_{n+1}^{-r} |Z_1=k \right] \p_k(Z_1=k)  = \gamma_k \mathbb{E}_k \left[ Z_{n}^{-r} \right],
$$ 
which proves that the sequence $(\mathbb{E}_k \left[ Z_{n}^{-r} \right] / \gamma_k^{n})_{n \in \mathbb{N}}$ is increasing.
We show that it is bounded. For $n \geq 1$ and $m \geq 0$, we have the following well-known branching property for $Z_n$:
\begin{equation}
\label{decomposition Zn1}
Z_{n+m} = \sum_{i=1}^{Z_m} Z_{n,i}^{(m)},
\end{equation}
where, under $\mathbb{P}_{\xi}$, the random variables $Z_{n,i}^{(m)}$ $( i \geq 1) $ are i.i.d., independent of $Z_m$, whose conditional probability law 
satisfies 
$ \mathbb{P}_{\xi} \left( Z_{n,i}^{(m)} \in \cdot \right)= \mathbb{P}_{T^m \xi} \left( Z_n \in \cdot \right) $, with $T^m$ the shift operator defined by $T^m (\xi_0, \xi_1 , \ldots ) = (\xi_m, \xi_{m+1} , \ldots )$. 
Intuitively, relation \eqref{decomposition Zn1} shows that, conditionally on $Z_m=i$, the annealed law of the process $Z_{n+m}$ is the same as that of a new process $Z_n $ starting with $i$ individuals.

Using  \eqref{decomposition Zn1} with $m=1$, the independence between $Z_1$ and ${Z}_{n,i}^{(1)}$ $(i \geq 1)$ and the fact that $\e_{i}Z_n^{-r} \leq \e_{k+1} Z_n^{-r}$ for all $i \geq k+1$, we have
\begin{eqnarray}
\label{eq induction 1}
\mathbb{E}_k \left[ Z_{n+1}^{-r} \right] &=& \mathbb{E}_k \left[ Z_{n+1}^{-r} | Z_1 = k\right] \p_k ( Z_1 = k) \notag \\
&& + \sum_{i=k+1}^{\infty} \mathbb{E} \left[ \left( \sum_{h=1}^{i} Z_{n,h}^{(1)} \right)^{-r} \bigg| Z_1 = i \right] \p_k ( Z_1 =i)  \nonumber \\
&\leq& \gamma_k \mathbb{E}_k \left[ Z_n^{-r} \right] +  \mathbb{E}_{k+1} \left[ Z_n^{-r} \right] .
\end{eqnarray}
We shall use the following change of measure: for $k \geq 1$ and $r>0$, let $\p_k^{(r)}$ be a new probability measure determined by
\begin{equation}
\label{changement de mesure}
\mathbb{E}_k^{(r)} [T] =\frac{\mathbb{E}_k \left[ \Pi_n^{-r} T \right]}{c_r^n}
\end{equation}
 for any $\mathcal{F}_n$-measurable random variable $T$, 
where $c_r = \e m_0^{-r}$. By \eqref{changement de mesure}, we obtain
\begin{equation}
\e_{k+1} \left[ Z_n^{-r} \right] = \e_{k+1}^{(r)} [ W_n^{-r} ] c_r^n,
\end{equation}
with $\sup_{n \in \mathbb{N}} \e_{k+1}^{(r)} [ W_n^{-r} ] = \e_{k+1}^{(r)} [ W^{-r} ]$ (see \cite{liu}, Lemma 2.1).
Moreover, we have $\e^{(r)} [ p_1^{k+1} (\xi_0) m_0^{r} ] = \gamma_{k+1}/ \e m_0^{-r}<1$ for any $r<r_{k+1}$.
So by Theorem \ref{thm harmonic moments W} we get
$
\e_{k+1}^{(r)} [ W^{-r} ] = C(r) < \infty 
$
and then $\e_{k+1} \left[ Z_n^{-r} \right] \leq C(r) c_r^n$ for any $r<r_k+ \varepsilon < r_{k+1}$.
Coming back to \eqref{eq induction 1} with $r< r_k+ \varepsilon $, we get by induction that
\begin{equation}
\label{induction j_0-1}
\mathbb{E}_{k} \left[ Z_{n+1}^{-r} \right] \leq \gamma_{k}^{n+1}  + C \sum_{j=0}^{n} \gamma_{k}^{n-j} c_r^j.
\end{equation}
Choose $r>r_k$ such that $c_r < \gamma_k$. Then, we have, as $ n \to \infty$,
\begin{equation}
\label{cv gamma < delta}
\frac{\mathbb{E}_{k} \left[ Z_{n+1}^{-r} \right]}{\gamma_k^{n+1}} \leq 1 + \frac{C}{\gamma_{k}} \sum_{j=0}^{n} \left( \frac{c_r}{\gamma_k} \right)^j \to \frac{C}{\gamma_{k} - c_r }.
\end{equation}
Thus the sequence $(\mathbb{E}_k \left[ Z_{n}^{-r} \right] / \gamma_k^{n})_{n \in \mathbb{N}}$ is bounded and \eqref{eq lem R Q} holds for any $r \in (r_k , r_{k} + \varepsilon)$. Using the fact that $\mathbb{E}_{k} \left[ Z_{n+1}^{-r'} \right] \leq \mathbb{E}_{k} \left[ Z_{n+1}^{-r} \right]$ for any $r'>r$, the result follows for any $r>r_k$, which ends the proof of the lemma.
\end{proof}
\begin{remark}
From the results stated above, with some additional analysis one can obtain the equivalent of the harmonic moments $\e Z_n^{-r}$ for any $r>0.$  
However, it is delicate to have an expression of the concerned constant in the equivalence.  This will be considered in
a forthcoming paper.
\end{remark}


Now we show that the radius of convergence $R$ of the power series $Q_k (t) = \sum_{j=k}^{\infty} q_{k,j} t^j$ is equal to $1$. Using the fact that $\sum_{j=k}^{\infty} \p_k \left( Z_n = j \right) = 1$,  part a) of Theorem \ref{thm small value probability 2} and the monotone convergence theorem, we have 
\[ \lim_{n \to \infty} \uparrow \gamma_k^{-n}  \sum_{j=k}^{\infty} \p_k \left( Z_n = j \right) = \sum_{j=k}^{\infty} q_{k,j} = + \infty,\]
which proves that $R \leq 1$.
We prove that $R=1$ by showing that $\sum_{j=k}^{+ \infty} j^{-r} q_{k,j} < \infty $ for $r>0$ large enough. Using 
part a) of Theorem \ref{thm small value probability 2}, the monotone convergence theorem and Lemma \ref{lem R Q}, we have,  for any $r>r_k$,
\begin{equation}
\sum_{j=k}^{+ \infty} j^{-r} q_{k,j} = \sum_{j=k}^{+ \infty} j^{-r} \lim_{n \to \infty} \uparrow \frac{\p_k (Z_n =j)}{\gamma_k^n}= 
 \lim_{n \to \infty}  \uparrow \frac{\e_k Z_n^{-r}}{\gamma_k^n} < \infty,
\end{equation}
which proves part b).

Now we prove part c) of Theorem \ref{thm small value probability 2}.
Using part a), the definition of $G_{k,n}$ and the monotone convergence theorem, we get \eqref{cv Qnk ->Qk}.
To prove the functional relation (\ref{relation Q_k}), recall that $G_{k,1} (t) = \sum_{j=k}^{\infty} p(k,j) t^j=\mathbb{E} f_0^k (t)$. 
By (\ref{relation rec qkj}) and Fubini's theorem, we get
\begin{eqnarray*}
 \gamma_k Q_k (t) &=& \sum_{j=k}^{\infty} \sum_{i=k}^{\infty} q_{k,i} \ p (i, j) \mathds{1}(i \leq j) t^j \\
              &=& \sum_{i=k}^{\infty} q_{k,i} \sum_{j=i}^{\infty} p (i, j) t^j \\
              &=& \sum_{i=k}^{\infty} q_{k,i} \mathbb{E} \left[ f_0^i (t)\right] \\
              &=& \mathbb{E} \left[ \sum_{i=k}^{\infty} q_{k,i}  f_0^i (t) \right] \\
              &=& \mathbb{E} \left[ Q_k (f_0 (t)) \right].
\end{eqnarray*}
This proves the functional relation (\ref{relation Q_k}). 

We now prove that the previous functional relation characterizes the function $Q_k$. 
To this end it suffices to show the unicity of the solution of (\ref{relation Q_k}).
Assume that there exists a power series $\hat Q(t)= \sum_{j=0}^{\infty} \hat q_{k,j} t^j$ on $[0,1)$ which verifies \eqref{relation Q_k} with the initial  condition $q_{k,k}=\hat q_{k,k}=1$. 
We first show by induction in $n$ that $\hat Q^{(n)} (0)=0$ for all $n \in \{ 0, \ldots , k-1 \}$.
Since $f_0 (0)=0$ and $\gamma_k \in (0,1)$, by \eqref{relation Q_k}, we get $\gamma_k \hat Q(0) = \hat Q (0)$, which implies that $\hat Q^{(0)}(0)=\hat Q(0)=0$. 
By  the induction hypothesis we have that $\hat Q^{(j)} (0)=0$ for all $j \in \{ 0, \ldots , n-1 \}$ for some $n < k-1.$ 
We show that $\hat Q^{(n)} (0)=0$. Using Fa\`a di Bruno's formula, we have
\begin{equation}
\label{Faa di Bruno}
\left( \hat Q \circ f_0 \right)^{(n)} (t) = \sum_{j=1}^n  \hat Q^{(j)}(f_0(t)) B_{n,j} \left( f_0^{(1)}(t), \ldots , f_0^{(n-j+1)} (t) \right),
\end{equation} 
where $B_{n,j}$ are the Bell polynomials,
defined for any $ 1 \leq j \leq n $ by 
\begin{eqnarray*}
\label{Bell}
 &&B_{n,j}(x_1,x_2,\dots,x_{n-j+1}) \\
 && \qquad \qquad =\sum{n! \over i_1!i_2!\cdots i_{n-j+1}!}
\left({x_1\over 1!}\right)^{i_1}\left({x_2\over 2!}\right)^{i_2}\cdots\left({x_{n-j+1} \over (n-j+1)!}\right)^{i_{n-j+1}},
\end{eqnarray*}
where the sum is taken over all sequences $(i_1, \ldots, i_{n-j+1})$ of non-negative integers such that $i_1 + \cdots + i_{n-j+1} = j$ and 
$i_1 + 2 i_2  + \cdots + (n-j+1)i_{n-j+1} = n$.
In particular $B_{n,n} (x_1) = x_1^n$.
Applying \eqref{Faa di Bruno} and using the fact that $f_0(0)=0$, $B_{n,n} \left( f_0^{(1)}(0)\right)= f_0^{(1)}(0)^{n}$ and $\hat Q^{(j)} (0)=0$ for all $j \in \{ 0, \ldots , n-1 \}$, we get
\begin{equation}
\label{derivee nieme Q(f(t))}
\left( \hat Q \circ f \right)^{(n)} (0) = \hat Q^{(n)} (0) \left( f_0^{(1)}(0) \right)^{n}.
\end{equation}
Then taking the derivative of order $n$ of both sides of \eqref{relation Q_k} and using \eqref{derivee nieme Q(f(t))}, we obtain that $\gamma_k \hat Q^{(n)} (0) = \hat Q^{(n)} (0) \gamma_{n}$ for $n<k-1$, which implies that $\hat Q^{(n)}(0)=0$.

Now we show that $\hat q_{k,j} = q_{k,j}$ for any $j \geq k+1$.
Using Fubini's theorem, the fact that $f_0, \ldots, f_{n-1}$ are i.i.d.\ and  iterating (\ref{relation Q_k}), we get
\begin{equation}
\label{relation Qk gn}
\mathbb{E} \left[  Q_k (\bar{g}_n (t)) \right] = \gamma_k^n  Q_k (t) \quad \text{and} \quad
\mathbb{E} \left[ \hat Q_k (\bar{g}_n (t)) \right] = \gamma_k^n \hat Q_k (t), 
\end{equation}
where $\bar{g}_n (t) = f_{n-1} \circ \ldots \circ f_0 (t) $.
By \eqref{relation Qk gn}, for all $t \in [0,1)$ and $n \in \mathbb{N}$, we have
\begin{eqnarray}
\label{Q1-Q2}
\left| Q_k(t)- \hat Q_k (t) \right| &=& \gamma_k^{-n} \left| \mathbb{E} \left[ Q_k(\bar{g}_n (t))- \hat Q_k(\bar{g}_n (t)) \right] \right| \notag \\
&=& \gamma_k^{-n} \left| \sum_{j=k}^{\infty} ( q_{k,j} - \hat q_{k,j} ) \mathbb{E} \left[ \bar{g}_n^j (t) \right] \right|\notag  \\
&\leq&  \gamma_k^{-n}   \sum_{j=k+1}^{\infty} \left| q_{k,j}- \hat q_{k,j} \right| G_{j,n} (t)  ,
\end{eqnarray}
where $G_{j,n} (t)$ is the generating function of $Z_n$ starting with $j$ individuals. 
To conclude the proof of the unicity it is enough to show that
 \begin{equation}
\lim_{n \to \infty}  \sum_{j=k+1}^{\infty} \left| q_{k,j}- \hat q_{k,j} \right|  \gamma_k^{-n}   G_{j,n} (t) = 0.
\label{final-001}
\end{equation}

We prove \eqref{final-001} using the Lebesgue dominated convergence theorem. 
Note that, by \eqref{cv Qnk ->Qk}, for all $n \in \mathbb{N}$, 
\begin{equation}
\gamma_j^{-n} G_{j,n} (t) \leq Q_j (t). 
\label{final-002}
\end{equation}
Therefore, using the fact that  $ \gamma_j < \gamma_k$ for all $j \geq k+1$, we have 
\begin{eqnarray*}
\lim_{n \to \infty} \gamma_k^{-n} G_{j,n} (t) 
= \lim_{n \to \infty} \left(\frac{\gamma_j}{ \gamma_k}  \right)^n \gamma_j^{-n} G_{j,n}(t)  
\leq \lim_{n \to \infty} \left(\frac{\gamma_j}{ \gamma_k}\right)^n Q_j (t)= 0,
\end{eqnarray*}
and $\gamma_k^{-n} G_{j,n} (t) \leq Q_j (t)$.
Now we show that $\sum_{j=k+1}^{\infty} \left| q_{k,j} - \hat q_{k,j} \right| Q_j (t) < \infty$ for all $t \in [0,1)$.
Indeed, by part b) of Theorem \ref{thm small value probability 2}, we have $\sum_{j=k}^{\infty} q_{k,j} j^{-r}< \infty$ for any $r>r_k$. 
In particular for a fixed $r>r_k$, there exists a constant $C>0$ such that, 
for all $j \geq 1$, $i \geq j$, it holds  $q_{j,i} \leq C i^{r}$. 
Therefore, 
\[
Q_j(t) \leq C \sum_{i=j}^{\infty} i^{r} t^i  \leq C t^{j} \sum_{i=0}^{\infty} (i+j)^{r} t^i \leq C \left( 2^r t^{j} \sum_{i=0}^{\infty} i^{r} t^i +  2^r t^{j} \sum_{i=0}^{\infty} j^{r} t^i \right)
\leq C_{r} j^r t^j.
\] 
Since $Q_k$ and $\hat Q_k$ are power series whose radii of convergence are equal to 1, we have, for any $t<1$,
\begin{equation}
\label{majoration Q}
\sum_{j=k+1}^{\infty} q_{k,j} Q_j (t) \leq  C_r  \sum_{j=k+1}^{\infty}  q_{k,j} j^r t^j < \infty,    
\quad \text{and} \quad \sum_{j=k+1}^{\infty} \hat q_{k,j} Q_j (t)
< \infty.
\end{equation}
Using the dominated convergence theorem, we see that 
$$ 
\lim_{n \to \infty} \gamma_k^{-n}   \sum_{j=k+1}^{\infty} \left| q_{k,j} - \hat q_{k,j} \right| G_{j,n} (t) =0.
$$ 
Therefore, from \eqref{Q1-Q2} we conclude that $Q_k (t)=\hat Q_k(t)$
for all $t \in [0,1)$.
This ends the proof of Theorem \ref{thm small value probability 2}.

\bibliographystyle{plain}
\bibliography{Bibliography_BPRE2}

\nocite{smith}

\nocite{athreya}

\nocite{athreya1971branching}

\nocite{athreya1971branching2}

\nocite{liu}

\nocite{bansaye2009large}

\nocite{boinghoff2010upper}

\nocite{bansaye2011upper}

\nocite{bansaye2012lower}

\nocite{kozlov2006large}

\nocite{bansaye2014small}

\end{document}